\documentclass{conm-p-l}
\usepackage{amssymb}
\usepackage{amsmath}
\usepackage{amsfonts}
\usepackage{amsthm}

\newcommand{\CM}{{\mathcal M}}

\newtheorem{theorem}{Theorem}[section]
\newtheorem{lemma}[theorem]{Lemma}
\newtheorem{proposition}[theorem]{Proposition}

\theoremstyle{definition}
\newtheorem{definition}[theorem]{Definition}
\newtheorem{example}[theorem]{Example}

\theoremstyle{remark}
\newtheorem{remark}[theorem]{Remark}

\numberwithin{equation}{section}

\begin{document}

\title[A Reduction Method for Higher Order Variational Equations]{A Reduction Method for Higher Order Variational Equations of Hamiltonian Systems}

\author{\rm{A.\, {\sc Aparicio\, Monforte}}
}
\address{XLIM, Universit\'{e}\, de\, Limoges, France}
%\curraddr{Department of Mathematics and Statistics,
%Case Western Reserve University, Cleveland, Ohio 43403}
\email{ainhoa.aparicio-monforte@unilim.fr}

\thanks{The first author was supported by a Grant from the Region Limousin (France).}  

\author{\rm{ J.-A.\, {\sc Weil}}
}
\address{XLIM, Universit\'{e}\, de\, Limoges, France}
\email{jacques-arthur.weil@unilim.fr}


%    General info
%\subjclass{34M03, 34M15, 34M25, 34Mxx, 20Gxx, 17B45, 17B80, 34A05, 34A26, 34A99}
  \renewcommand{\subjclassname}{%
    \textup{2010} Mathematics Subject Classification}
% OLDER : 34M03, 34M15, 34M25, 34Mxx, 20Gxx, 17B45, 17B80, 34A05, 34A26, 34A99}

\date{June 2010 and, in revised form, Oct 12, 2010.}



\begin{abstract}
Let $\mathbf{k}$ be a differential field and let $[A]\,:\,Y'=A\,Y$ be a linear differential system where $A\in\mathrm{Mat}(n\,,\,\mathbf{k})$. We say that $A$ is in a reduced form if $A\in\mathfrak{g}(\bar{\mathbf{k}})$ where $\mathfrak{g}$ is the Lie algebra of $[A]$ and $\bar{\mathbf{k}}$ denotes the algebraic closure of $\mathbf{k}$. We owe the existence of such reduced forms to a result due to Kolchin and Kovacic \cite{Ko71a}.
This paper is devoted to the study of reduced forms,  of (higher order) variational equations along
a particular solution of a complex analytical hamiltonian system $X$.
Using a previous result \cite{ApWea}, we will assume that the first order
variational equation has an abelian Lie algebra so that, at first order, there are no Galoisian obstructions to Liouville integrability.
We give a strategy to (partially) reduce the variational equations at order $m+1$ if the variational equations at order $m$ are already in a reduced form and their Lie algebra is abelian.
Our procedure stops when we meet obstructions to the meromorphic integrability of $X$.
We make strong use both of the lower block triangular structure of the variational equations and of the notion of associated Lie algebra of a linear differential system (based on the works of Wei and Norman in \cite{WeNo63a}).
Obstructions to integrability appear when at some step we obtain a non-trivial commutator between
a diagonal element and a nilpotent (subdiagonal) element of the associated Lie algebra.
We use our method coupled with a reasoning on polylogarithms to give a new and systematic proof of the non-integrability of the H\'enon-Heiles system.
We conjecture that our method is not only a partial reduction procedure but a complete reduction algorithm. In the context of complex Hamiltonian systems, this would mean that our method would be an  effective version of the Morales-Ramis-Sim\'o theorem.
\end{abstract}

\maketitle

%\tableofcontents
\section{Introduction}

Let $(\mathbf{k}\,,\, '\,)$ be a differential field and let $[A]: \; Y'=AY$ be a linear differential system with $A\in \mathcal{M}_{n}(\mathbf{k})$.
We say that the system is in {\em reduced form} if its matrix can be decomposed as $A=\sum^{d}_{i=1} \alpha_i A_i$ where $\alpha_i \in \mathbf{k}$ and $A_i\in Lie(Y'=AY)$, the Lie algebra of the differential Galois group of $[A]$.

This notion of reduced form was introduced in \cite{Ko71a} and subsequently used (for instance \cite{MiSi96a} and \cite{MiSi96b}) to study the inverse problem.
It has been revived, with a constructive emphasis, in \cite{ApWea}. It is a powerful tool in various aspects of linear differential systems.
The main contribution of this work lies in the context of  Hamiltonian mechanics and Ziglin-Morales-Ramis theory \cite{MoRaSi07a}: reduced forms provide a new and powerful effective method to obtain (non-)abelianity and integrability obstructions from higher variational differential equations.

This article is structured in the following way.
First we lay down the background on Hamiltonian systems, differential Galois theory, integrability and Morales-Ramis-Sim\'o theorem. In section \ref{section: reduced forms}, we define precisely the notions of reduced form and Wei-Norman decomposition and the link between them.
Section \ref{section: reduced VEm} contains the theoretical core of this work: we focus on the application of reduced forms to the study of the meromorphical integrability of Hamiltonian systems. We introduce a reduction method for block lower triangular linear differential systems and apply it to higher variational equations, in particular when the Lie algebra of the diagonal blocks is abelian and of dimension 1.
 In section \ref{section: new proof}, we demonstrate the use of  this method, coupled with our reduction algorithm for matrices in $\mathfrak{sp}(2,\mathbf{k})$ \cite{ApWea}  by giving a new, effective and self-contained Galoisian non-integrability proof of the degenerate H\'enon-Heiles system (\cite{Mo99a} ,\cite{MoRaSi07a}, \cite{MaSi09a}) which has long served as a key example in this field.

\section{Background}
\subsection{Hamiltonian Systems}

Let $(M\,,\,\omega)$ be a complex analytic symplectic manifold of complex dimension $2n$ with $n\in\mathbb{N}$. Since $M$ is locally isomorphic to an open domain $U\subset\mathbb{C}^{2n}$, Darboux's theorem allows us to choose a set of local coordinates  $(q\,,\,p)=(q_1 \,\ldots q_n\,,\, p_1\ldots p_n)$ in which the symplectic form $\omega$ is expressed as $J:=\tiny\left[\begin{array}{cc}0 & I_n \\-I_n & 0\end{array}\right]$.  In these coordinates, given a function $H\in C^{2}(U)\,:\,U\,\longrightarrow\,\mathbb{C}$ (the Hamiltonian) we define a Hamiltonian system over $U\in\mathbb{C}^{2n}$, as the differential equation given by the vector field $X_H:= J\nabla H$:
\begin{equation}\label{(1)}
\begin{array}{cccc}
\dot{q}_i = \frac{\partial H }{\partial p_i}(q\,,\,p) &,& \dot{p}_i = -\frac{\partial H }{\partial q_i}(q\,,\,p)& \text{for} \,\, i=1\ldots n
\end{array}
\end{equation}

The Hamiltonian $H$ is constant over the integral curves of (\ref{(1)}) because $X_H\cdot H:=\langle \nabla H\,,\, X_H\rangle = \langle \nabla H \,,\, J\nabla H\rangle =0$. Therefore, integral curves lie on the energy levels of $H$.
A function $F\,:\, U\, \longrightarrow \,\mathbb{C}$ meromorphic  over $U$ is called a \emph{meromorphic first integral of } (\ref{(1)})   if it is constant over the integral curves of (\ref{(1)})  (equivalently  $X_H \cdot F =0$). Observe that the Hamiltonian is a first integral of (\ref{(1)}).

The Poisson bracket $\lbrace \,,\,\rbrace$ of two meromorphic functions $f, g$ defined over a symplectic manifold, is defined by $\lbrace f \,,\, g \rbrace:=\langle \nabla f \,,\, J\nabla g\rangle$; in  the  Darboux coordinates its expression is $\lbrace f \,,\, g \rbrace = \sum^{n}_{i=1} \frac{\partial f}{\partial q_i}\frac{\partial g}{\partial p_i}-\frac{\partial f}{\partial p_i}\frac{\partial g}{\partial q_i}$. The Poisson bracket endows the set of first integrals with a structure of Lie algebra. A function $F$ is a first integral of (\ref{(1)})  if and only if $\lbrace F\,,\, H\rbrace=0$  (i.e $H$ and $F$ are \emph{in involution}).

A Hamiltonian system with $n$ degrees of freedom, is called \emph{meromorphically Liouville integrable} if it possesses $n$ first integrals (including the Hamiltonian) meromorphic over $U$ which are functionally independent and in pairwise involution.

\subsection{Variational equations}\label{subsection:variational equations}
Among the various  approaches to the study of meromorphic integrability of complex Hamiltonian systems,
 we choose a  Ziglin-Morales-Ramis type of approach. Concretely, our starting points are the Morales-Ramis \cite{Mo99a} Theorem and its generalization, the Morales-Ramis-Sim\'o Theorem \cite{MoRaSi07a}. These two results give necessary conditions for the meromorphic integrability of Hamiltonian systems. We need to introduce here the notion of variational equation of order $m\in\mathbb{N}$ along a non punctual integral curve of (\ref{(1)}).

Let $\phi(z,t)$ be the flow defined by the equation (\ref{(1)}).
For $z_0 \in \Gamma$, we let $\phi_{0}(t):=\phi(z_0 \,,\, t)$  denote a temporal parametrization of a non punctual integral curve $\Gamma$ of (\ref{(1)})  such that $z_0 = \phi(w_0 , t_0)$.
We define  $\mathrm{(VE^{m}_{\phi_0})}$ the \emph{$m^{th}$ variational equation} of (\ref{(1)})  along $\Gamma$ as the differential equation satisfied by the $\xi_{j}:=\frac{\partial^{j}\phi(z\,,\,t)}{\partial z^j}$ for $j\leq m$. For instance, $\mathrm{(VE^{3}_{\phi_0})}$ is given by (see \cite{Mo99a} and \cite{MoRaSi07a}):
\begin{eqnarray}
\nonumber \dot{\xi}_1 &=& d_{\phi_0} X_H \xi_1\\
\nonumber \dot{\xi}_2 &=& d^{2}_{\phi_0}X_H(\xi_1\,,\, \xi_1) + d_{\phi_0}X_H \xi_2\\
\nonumber \dot{\xi}_2 &=& d^{3}_{\phi_0}X_H(\xi_1\,,\, \xi_1\,,\, \xi_1) + 2 d^2_{\phi_0}X_H (\xi_1\,,\,\xi_2) + d_{\phi_0} X_H \xi_3.
\end{eqnarray}
For $m=1$, the equation $\mathrm{(VE^{1}_{\phi_0})}$ is a linear differential equation
$$\dot{\xi}_1 = A_{1} \xi_1\text{   where   }A_{1}:= d_{\phi_0} X_H= J\cdot Hess_{\phi_0}(H)\in\mathfrak{sp}(n\,,\, \mathbf{k})\text{  and  }
	\mathbf{k}:=\mathbb{C}\langle \phi_0(t) \rangle.$$
Higher order variational equations are not linear in general for $m\geq 2$. However, taking symmetric products,
one can give for every $\mathrm{(VE^{m}_{\phi_0})}$ an equivalent linear differential system  $\mathrm{(LVE^{m}_{\phi_0})}$ called the {\em linearized} $m^{th}$ {\em variational equation} (see \cite{MoRaSi07a}).

Since the $\mathrm{(LVE^{m}_{\phi_0})}$ are linear differential systems, we can consider them under the light of differential Galois theory (\cite{PuSi03a,Mo99a}).
We take as base field the differential field $\mathbf{k} := \mathbb{C}\langle \phi_{0} \rangle$ generated by the coefficients of $\phi_{0}$ and their derivatives.
Let  $K_m$  be a Picard Vessiot extension of $\mathrm{(LVE^{m}_{\phi_0})}$  for $m\geq 1$.
The differential Galois group $G_m := \text{Gal}(K_m /\mathrm{k} )$ of  $\mathrm{(LVE^{m}_{\phi_0})}$ is the group of all differential automorphisms of $K_m$ that leave the elements of $\mathbf{k}$ fixed.

 As $G_m$ is isomorphic to a algebraic linear group over $\mathbb{C}$, it is in particular an algebraic manifold and we can define its Lie algebra
 $\mathfrak{g}_m:=T_{I_{d_m}} G^{\circ}_m$, the tangent space of $G_m$ at $I_{d_m}$ (with $ d_m= \tiny\sum^{m}_{i=1} \binom{n+i-1}{n-1}$ the size of $\mathrm{(LVE^{m}_{\phi_0})}$). The Lie algebra $\mathfrak{g}_m$ is a complex vector space of square matrices of size $d_m$ whose Lie bracket is given by the commutator of matrices $[M\,,\,N] = M\cdot N - N \cdot M$.
 We say that $\mathfrak{g}_m$ is abelian if $[\mathfrak{g}_m\,,\, \mathfrak{g}_m] = 0$.

Following the notations above, we can finally give the Morales-Ramis-Sim\'o theorem:
\begin{theorem}\label{MRS}(\cite{MoRaSi07a}): If the Hamiltonian system (\ref{(1)})  is meromorphically Liouville integrable
then the $\mathfrak{g}_m$ are abelian for all $m\in \mathbb{N}^{\star}$.
\end{theorem}

Partial effective versions of this theorem have been proposed. In  \cite{MoRaSi07a} (and already \cite{Mo99a}), a local criterion is given for the case when the first variational equation has Weierstrass functions as coefficients~; in
\cite{MaSi09a}, a powerful approach using certified numerical computations is proposed. In the case of Hamiltonian systems with a homogeneous potential, yet another approach is given in \cite{CaDuMaPr10a}. \\
Our aim is to propose an alternative (algorithmic) method using a (constructive) notion of reduced form for the variational equation. This strategy should supply new criteria of non-integrability as well as some kind of ``normal form along a solution''. We will now explain this notion of reduced form (which we started investigating in
\cite{ApWea}) and show how to apply it. We will then apply our reduction method in detail on the well-known degenerated case of the Henon-Heiles system proposed in \cite{MoRaSi07a}.

\section{Reduced Forms}\label{section: reduced forms}

Let $(\mathbf{k}\,,\,'\,)$ be a differential field with field of constants $C$ and let $Y'=AY$ be a linear differential system with $A=(a_{i j})\in \mathcal{M}_{n}(\mathbf{k})$.
Let $G$ be the differential Galois group of this system and $\mathfrak{g}$ the Lie algebra of $G$.
We sometimes use the slight notational abuse $\mathfrak{g}=Lie(Y'=AY)$.
\\

Let $a_{1},\ldots,a_{r}$ denote a basis of the $C$-vector space spanned by the entries $a_{i,j}\in k$ of $A$.
Then we have
$$A:=\sum^{r}_{i=1} a_{i}(x) M_i ,\quad M_i \in\mathcal{M}_{n}(C).$$
This decomposition appears (slightly differently) in \cite{WeNo63a}, we call it a \emph{Wei-Norman decomposition} of $A$.
Although this decomposition is not unique (it depends on the choice of the basis $(a_{i})$), the $C-$vector space generated by the $M_i$ is unique.
\begin{definition} With these notations, the Lie algebra generated by $M_1 ,\ldots , M_r$ and their iterated Lie brackets is called \emph{the Lie algebra associated to $A$}, and will be denoted as $Lie(A)$.
\end{definition}

\begin{example} Consider the matrix $$
A_1:=\left[\begin{array}{cccc} 0 & 0 & 2/x & 0 \\ 0 & 0 & 0 & 2/x\\
\frac{2(x^4 - 10 x^2 + 1 )}{x(x^2 + 1)^2} & 0 & 0 & 0\\ 0 & -\frac{12 x }{(x^2 + 1)^2 } & 0 & 0\end{array}\right].
$$
Expanding the fraction $\frac{2(x^4 - 10 x^2 + 1 )}{x(x^2 + 1)^2}$ gives a Wei-Norman decomposition as $$A_{1}=\frac2{x} M_{1}
-\frac{12 x }{(x^2 + 1)^2 } M_{2},$$
where
$$M_{1}= \left[\begin{array}{cccc} 0 & 0 & 1 & 0 \\ 0 & 0 & 0 & 1\\
1 & 0 & 0 & 0\\ 0 & 0 & 0 & 0\end{array}\right],\,
M_{2}=\left[\begin{array}{cccc} 0 & 0 & 0 & 0 \\ 0 & 0 & 0 & 0\\
2 & 0 & 0 & 0\\ 0 & 1 & 0 & 0\end{array}\right]$$
and $Lie(A_{1})$ has dimension $6$.
\end{example}

A celebrated theorem of Kovacic (and/or Kolchin) states that $\mathfrak{g}\subset Lie(A)$. This motivates the following definition~:
\begin{definition}
We say that $A$ is in \emph{reduced form} if $Lie(A)=\mathfrak{g}$.
\end{definition}

A {\em gauge transformation} is a change of variable $Y=PZ$ with $P\in \mathrm{GL}(n\,,\, \mathbf{k}) $.
Then $Z'=RZ$ where $R:=P^{-1} (AP-P')$. In what follows, we adopt the notation $P[A]:=P^{-1} (AP-P')$
for the system obtained after the gauge transformation $Y=PZ$.\\

 The following theorem due to Kovacic (and/or Kolchin) ensures the existence of a gauge transformation $P\in \mathrm{GL}(n\,,\, \bar{\mathbf{k}}) $ such that
 $P[A]\in\mathfrak{g}(\bar{\mathbf{k}})$ when $ {\mathbf{k}} $ is a  $C_1$-field\footnote{A field $k$ is called a $C_1$-field (or cohomologically trivial) if any homogeneous polynomial $P\in k[X_1,\ldots,X_n]_{=d}$ of degree $d$ has a non-trivial zero in $k^n$ when $n>d$, i.e the number of variables is bigger than the degree.
 All differential fields of coefficients considered in this article will belong to the $C_1$ class.}

\begin{theorem}[see\cite{Ko71a,PuSi03a} p.25 Corollary 1.32 ]\label{Kovacic}
Let $k$  be a differential $C_1$-field. Let $A\in\mathcal{M}_n (k)$ and assume that the differential Galois Group $G$ of the system $Y'=AY$ is connected. Let $\mathfrak{g}$ be the Lie algebra of $G$. Let $H$ be a connected algebraic group such that
its Lie algebra $\mathfrak{h}$  satisfies $A\in\mathfrak{h}(k)$. Then $G\subset H$ and there exists $P\in H(k)$ such that the equivalent  differential equation $F'=\tilde{A}F$, with $Y=PF$ and $\tilde{A}=P[A]=P^{-1}AP-P^{-1}P'$, satisfies $\tilde{A}\in \mathfrak{g}(k)$.
\end{theorem}

We say that a matrix $P\in\mathrm{GL}_n (\mathbf{k})$ is a \emph{reduction matrix} if $P[A]\in\mathfrak{g}(\mathbf{k})$, i.e $P[A]$ is in reduced form.
We say that a matrix $Q\in\mathrm{GL}_n (\mathbf{k})$ is a \emph{partial reduction matrix} when $Q[A]\in\mathfrak{h}(\mathbf{k})$ with $\mathfrak{g}\subsetneq\mathfrak{h}\subsetneq Lie(A)$. The general method used to put $A$ in a reduced form consists in performing successive partial reductions until a reduced form is reached. \\

In our paper \cite{ApWea}, we provide a reduction algorithm that computes a reduction matrix $P_1 \in\mathrm{Sp}(2,\mathbf{k})$ for $4\times 4$ linear differential systems $Y'=A_1Y$ with $A_1\in\mathfrak{sp}(2,\mathbf{k})$
(and also for $2\times 2$ systems). The first variational equation of a Hamiltonian system with $n=2$ degrees of freedom belongs  to this class of systems.  If $P_1$ is a reduction matrix for $A_1$ then $Sym^{m} P_1$ is a reduction matrix for $sym^{m} A_1$ because $Sym^m $ is a group morphism (see \cite{PuSi03a}, chapter 2 or \cite{FuHa91a} appendix B2).\\
In what follows, we will assume that we have reduced the first variational equation, that its Lie algebra is abelian (so that the Morales-Ramis theorem gives no obstruction to integrability), and use this to start reducing higher variational systems.\\

 We will follow the philosophy of Kovacic's theorem \ref{Kovacic} and look for reduction matrices inside $\exp(Lie(A))$.
 We remark that, in the context of Lie-Vessiot systems, an analog of the above Kolchin-Kovacic reduction theorem is  given by Blazquez and Morales (\cite{BlaMor10}, section 5, in particular theorems 5.3 and 5.8) in relation to Lie reduction.\\

The notion of a reduced form is useful in many contexts, such as: inverse problems (where the notion was first studied), the computation of  the transcendence degree of  Picard Vessiot extensions, fast resolution of linear differential systems with an  abelian Lie algebra and to implement the Wei-Norman method for solving linear differential systems with a solvable Lie algebra (using the Campbell-Hausdorff formula) \cite{WeNo63a}.  Reduced forms are also a new and powerful tool that provides (non-)abelianity and integrability obstructions for (variational) (see Theorem \ref{MRS}) linear differential equations arising from Hamiltonian mechanics, as we will now see.

\section{Reduced Forms for Higher Variational Equations}\label{section: reduced VEm}
\subsection{Preliminary results}\label{preliminary}
	
 Let $(\mathbf{k}\,,\, ' )$ be a differential field and let $d\in\mathbb{N}$.  Consider a linear differential system $Y' = AY$ whose matrix $A\in \CM_{d}(\mathbf{k})$ is block lower triangular  as follows:
 \begin{equation}
 A:=\left[\begin{array}{cc}A_{1} & 0 \\ A_{3} & A_{2}\end{array}\right]= A_{diag} + A_{sub} \text{   where    } A_{diag}=\left[\begin{array}{cc} A_1 & 0 \\ 0 & A_2\end{array}\right] \text{   and   } A_{sub}=\left[\begin{array}{cc} 0 & 0 \\ A_3 & 0\end{array}\right].
 \end{equation}
The submatrices satisfy $ A_{1}\in \CM_{d_1}(\mathbf{k})$, $ A_{2}\in \CM_{d_2}(\mathbf{k})$, $A_3\in \CM_{d_2 \times d_1} (\mathbf{k})$ and their dimensions add-up $d=d_1 + d_2$.

Let $${\CM }_{diag}:=\left\{\left[\begin{array}{cc}A_{1} & 0 \\ 0 & A_{2}\end{array}\right], A_{i}\in\CM_{d_{i}}(\mathbf{k})\right\} $$
and $${\CM }_{sub}:=\left\{\left[\begin{array}{cc}0 & 0 \\ B_{1} & 0\end{array}\right], B_{1}\in\CM_{d_{2}\times d_{1}}(\mathbf{k})\right\} $$
\begin{lemma}\label{diagsub}
Let $M_{1},M_{2}\in {\CM }_{diag}$ and $N_{1},N_{2}\in {\CM }_{sub}$. Then
$M_{1}.M_{2}\in  {\CM }_{diag}$,
$N_{1}.N_{2}=0$ (so that $N_{1}^{2}=0$ and $\exp(N_{1})=Id + N_{1}$),
and $[M_{1},N_{1}]\in {\CM }_{sub}$.
\end{lemma}
The proof is a simple linear algebra exercise.\\

\noindent Let $\mathfrak{g}:=Lie(Y'=AY)$ be the Lie algebra of the Galois group of $Y'=AY$ and let $\mathfrak{h}:=Lie(A)$ denote the Lie algebra associated to $A$.
We write $\mathfrak{h}_{diag}:=\mathfrak{h} \cap {\CM }_{diag}$ and $\mathfrak{h}_{sub}:=\mathfrak{h} \cap {\CM }_{sub}$.
The lemma shows that they are both Lie subalgebras (with $\mathfrak{h}_{sub}$ abelian) and
$\mathfrak{h}=\mathfrak{h}_{diag}\oplus \mathfrak{h}_{sub}$. Furthermore, $[\mathfrak{h}_{diag},\mathfrak{h}_{sub}]\subset \mathfrak{h}_{sub}$
(i.e $\mathfrak{h}_{sub}$ is an ideal in $\mathfrak{h}$). When $\mathfrak{h}_{diag}$ is abelian, obstructions to the abelianity of $\mathfrak{h}$
only lie in the brackets $[\mathfrak{h}_{diag},\mathfrak{h}_{sub}]$.

\subsection{A first partial reduction for higher variational equations}\label{subsection: first partial reduction}

Using the algorithm of \cite{ApWea}, we may assume that the first variational equation has been put into a reduced form. We further assume that the first variational equation has an abelian Lie algebra (so that there is no obstruction to integrability at that level).\\

As stated in section \ref{subsection:variational equations},  each $\mathrm{(VE^{m}_{\phi_0})}$ is equivalent to a linear differential system $\mathrm{(LVE^{m}_{\phi_0})}$ whose matrix we denote by $A_m$. The structure of the $A_m$ is block lower triangular , to wit
\begin{equation}
A_m :=\left[\begin{array}{cc} sym^{m}(A_1) & 0 \\ B_m & A_{m-1}\end{array}\right]\in M_{d_m} (\mathbf{k})
\end{equation}
where  $A_1$ is the matrix of  $\mathrm{(LVE^{1}_{\phi_0})}$. Assume that $A_{m-1}$ has been put in reduced form by
a reduction matrix  $P_{m-1}$. Then the matrix $Q_m \in \mathrm{GL}(d_m\,,\,\mathbf{k})$ defined by
$$
Q_m:=\left[\begin{array}{cc} Sym^{m}(P_1) & 0 \\ 0 & P_{m-1}\end{array}\right]
$$
puts the diagonal blocks of the matrix $A_m$ into a reduced form  (i.e the system would be in reduced form if there were no $B_{m}$)
and preserves the block lower triangular  structure. Indeed,
$$
Q_m [A_m] = \left[ \begin{array}{cc} Sym^m(P_1)[sym^m A_1] & 0 \\ \tilde{B}_{m} & P_{m-1}[A_{m-1}]\end{array}\right] $$
  where
  $$ \tilde{B}_{m}:=P^{-1}_{m-1} B_m Sym^{m}(P_1).
$$
Applying the notations of the previous section to $\tilde{A}:=Q_m [A_m]$, we see that $Lie(\tilde{A})_{diag}$ and $Lie(\tilde{A})_{sub}$ are abelian.
Obstructions to integrability stem from brackets between the diagonal and subdiagonal blocks. To aim at a reduced form, we need
transformations which ``remove'' as many subdiagonal terms as possible while preserving the (already reduced) diagonal part.
Recalling Kovacic's theorem \ref{Kovacic}, our partial reduction matrices will arise as exponentials from subdiagonal elements.

\subsection{Reduction tools for higher variational equations}
\begin{proposition}\label{partial reduction}
Let $A:=Q_m [A_m]$ as above be the matrix of the $m$-th variational equation $Y'=AY$ after reduction of the diagonal part.
Write $A=A_{diag} + \sum^{d_{sub}}_{i=1} \beta_i B_i$ with $\beta_i \in\mathbf{k}$, where the $B_{i}$ form a basis of $Lie(A)_{sub}$
(in the notations of section \ref{preliminary}).
\\
Let $[A_{diag}\,,\, B_1]=\sum_{i=1}^{{d_{sub}}} \gamma_{i} B_{i}$, $\gamma_{i} \in\mathbf{k}$.
Assume that the equation $y'=\gamma_{1}y+\beta_{1}$ has a solution $g_{1}\in k$. Set  $P:=\exp(g_{1}B_{1})=(Id+g_{1}B_{1})$.
Then $$ P[A] = A_{diag} + \sum_{i=\bf{2}}^{d_{sub}} \left[\beta_{i}+g_{1}\gamma_{i}\right]B_{i}, $$
i.e $P[A]$ no-longer has any terms in $B_{1}$.
\end{proposition}

\begin{proof} Recall that $P[A] = P^{-1}(AP-P')$ and
let $P=Id + g_1 B_1$. We have  $P'=g'_1 B_1$ whence
$$AP=(A_{diag} + \sum^{d_{sub}}_{i=1} \beta_i B_i)(I+g_{1} B_1) = A_{diag} + \sum_{i\geq 1} \beta_i B_i + g_{1} A_{diag} B_1$$
since $B_i B_j = 0$.  Therefore we have
 $AP-P' = A_{diag}+ g_{1} A_{diag} B_1 + (\beta_{1}-g_{1}')B_{1} + \sum^{d}_{i=2}\beta_i B_i$ which implies
\begin{eqnarray}
\nonumber P^{-1}(AP-P') &=& (Id-g_{1} B_1) \left[A_{diag}+ g_{1} A_{diag} B_1 + (\beta_{1}-g_{1}')B_{1} + \sum^{d}_{i=2}\beta_i B_i \right]\\
\nonumber&=&A_{diag} + g_{1}[A_{diag}\,,\, B_1] +  (\beta_{1}-g_{1}')B_{1} + \sum^{d}_{i=2}\beta_i B_i
\end{eqnarray}
 because $B_1 A_{diag} B_1 = B_1 [A_{diag} \,,\, B_1] + A_{diag} B_{1} B_{1} = B_{1}\left[\sum \gamma_i B_i\right]=0$.
 So,  as  $g_{1}'=\gamma_{1}g_{1}+\beta_{1}$, we obtain
 $$	
 P[A] = A_{diag} + \sum_{i=\bf{2}}^{d_{sub}} \left[\beta_{i}+g_{1}\gamma_{i}\right]B_{i}.
 $$

\end{proof}

\begin{remark} If $\gamma_{1}=0$ then we simply have $g_{1}=\int\beta_{1}$.
In that case, suppose that $\mathbf{k}=\mathbb{C}(x)$ and that $\beta_1 = R'_1 + L_1$ where  $R_1 \in \mathbb{C}(x)$ and $L_1 \in \mathbb{C}(x)$ has only simple poles, then $\int \beta_1 \notin \mathbb{C}(x)$. However, if we apply proposition \ref{partial reduction} with the change of variable $Y= (I + R_1 B_1) Z$  a term in $B_1$ will be left that will only contain simple poles.
\end{remark}

This proposition gives a nice formula for reduction. However, it is hard to i\-terate unless $Lie(A)$ has additional properties (solvable, nilpotent, etc)
because the next iteration may ``re-introduce'' $B_{1}$ in the matrix (because of the expression of the brackets).
This proposition provides a reduction strategy when  the map $[A_{diag},.]$ admits a triangular representation.
\\

To achieve this, we specialize to the case when the Lie algebra $\mathfrak{g}_{diag}$ has dimension (at most) $1$.
Then we have $A_{diag} = \beta_{0} A_{0}$ where $\beta_{0}\in k$ and $A_{0}$ is a constant matrix. The above proposition specializes nicely~:

\begin{example}
If $A_{diag}=\beta_0 A_0$ with $\beta_0\in\mathbf{k}$, $A_{0}\in\mathcal{M}_{n}(\mathbb{C})$ and $[A_0\,,\, B_1] = \lambda B_1$ for some constant eigenvalue $\lambda \neq 0$ then the change of variable $Y=PZ$ with $P:=(Id + g B_1)$, with $g'= \lambda g\beta_0 + \beta_1$, satisfies
$P[A] = \beta_0 A_0 + \sum^{d_{sub}}_{i\geq 2} \beta_i B_i$.
\end{example}

\noindent To implement this (and obtain a general reduction method), we let $\Psi_{0} : \mathfrak{h}_{sub} \rightarrow \mathfrak{h}_{sub}$, $B\mapsto [A_{0},B]$.
This is now an endomorphism of a finite dimensional vector space~; up to conjugation, we may assume the basis $(B_{i})$ to be
the basis in which the matrix of $\Psi_{0}$ is in Jordan form. We are then in position to apply the proposition iteratively (see the example below for details on the process).

\begin{remark}
Not that $A_0$ needs not be diagonal. The calculations of lemma \ref{partial reduction} and subsequent proofs remain valid when $A_0$ is block lower triangular.
\end{remark}
We have currently implemented this in Maple for the case when $A_{diag}$ is monogenous, i.e. its associated Lie algebra has dimension $1$. We will show the power of this method and of the implementation by giving a new proof of non-integrability of
the degenerate Henon-Heiles system whose first two variational equations are abelian but which is not integrable.

\section{A new proof of the non integrability of a degenerate H\'enon-Heiles system}\label{section: new proof}

In this section we consider the following H\'enon Heiles Hamiltonian  \cite{Mo99a}, \cite{MoRaSi07a},
\begin{equation}\label{HH}
H:=\frac{1}{2}(p^2_1 + p^2_2) + \frac{1}{2}(q^2_1 + q^2_2) + \frac{1}{3}q^3_1 + \frac{1}{2}q_1 q^2_2
\end{equation}
as given in  \cite{Mo99a}. This Hamiltonian's meromorphic non integrability was proved in \cite{MoRaSi07a}.
The Hamiltonian field is
$$
\dot{q}_1 = p_1 \,,\, \dot{q}_2 = p_2 \,,\, \dot{p}_1 = -q_1 (1+q_1) - \frac{1}{2}q^2_2 \,,\,\dot{p}_2 = -q_2 (1+ q_1 ).
$$
This degenerate H\'enon Heiles system was an important test case which motivated \cite{MoRaSi07a}. Its non integrability was reproved in \cite{MaSi09a} to showcase the method used by the authors. We follow in this tradition by giving yet another proof using our systematic method.
Our reduction provides a kind of ``normal form along $\phi$'' in addition to a non integrability proof.
The readers wishing to reproduce the detail of the calculations will find a Maple file at the url
\begin{center}\begin{verbatim}Êhttp://www.unilim.fr/pages_perso/jacques-arthur.weil/charris/ 
\end{verbatim}\end{center}
It contains the commands needed to carry on the reduction of the $\mathrm{(LVE^{m}_{\phi})}$ for $i=1\ldots 3$. The reduction of $\mathrm{(LVE^{3}_{\phi})}$ may take several minutes to complete.
\subsection{Reduction of $\mathrm{(VE^{1}_{\phi})}$}
On the invariant manifold $\lbrace q_2 = 0 \,,\, p_2 =0 \rbrace$
we consider the non punctual particular solution
$$\phi(t) = \left(\,\frac{3}{2}\frac{1}{\cosh(t/2)^2} - 1 \,,\, 0\,,\, -\frac{3}{2}\frac{\sinh(t/2)}{\cosh(t/2)^3}\,,\,0\right).$$
and the base field is  $\mathbf{k}=\mathbb{C}\langle \phi \rangle = \mathbb{C}(e^{t/2})$.
Performing the change of independent variable $x=\mathrm{e}^{t/2}$, we obtain  an equivalent system with coefficients in $\mathbb{C}(x)$ given by
$$
A_1:=\left[\begin{array}{cccc} 0 & 0 & 2/x & 0 \\ 0 & 0 & 0 & 2/x\\
\frac{2(x^4 - 10 x^2 + 1 )}{x(x^2 + 1)^2} & 0 & 0 & 0\\ 0 & -\frac{12 x }{(x^2 + 1)^2 } & 0 & 0\end{array}\right].
$$
\\
\vspace{0.5cm}
Applying the reduction algorithm from \cite{ApWea} we obtain the reduction matrix
$$ P_1:=\left[\begin{array}{cccc} -\frac{6(x-1)(x+1)x^2}{(x^2 + 1)^3} & 0 & -\frac{x^{10} + 15 x^8 - 16 x^6 - 144x^4+15x^2+1}{12x^2 (x^2 +1)^3} & 0\\ 0 &\frac{x^4-4x^2 + 1}{(x^2 + 1)^2} & 0 & -\frac{5 x^4 + 16 x^2 - 13}{3(x^2 + 1 )^2}\\  \frac{6 x^2 (x^4 - 4x^2 + 1)}{(x^2 + 1) ^4} & 0 & -\frac{x^{12} + 4 x^{10} + 121 x^8 + 256 x^6 - 249 x^4 - 4 x^2 - 1}{12 x^2 (x^2 + 1) ^4} & 0 \\ 0 & \frac{6(x^2 - 1)x^2}{(x^2 + 1)^3} & 0 & \frac{x^6 - x^4 - 17 x^2 + 1}{(x^2 + 1)^3} \end{array}\right]$$
  that yields the reduced form
  $$ A_{1,R}=\frac{5}{3 x}\left[\begin{array}{cccc}0 &0 & 1 & 0\\ 0 & 0 & 0& 6/5 \\ 0 & 0 & 0 & 0 \\ 0 & 0 & 0 & 0 \end{array}\right].
$$
We see that $\mathrm{dim}_{\mathbb{C}}\left(Lie(A_{1,R})\right) =1$ and since $\frac{5}{3 x}$ has one single pole, we cannot further reduce without extending the base field $\mathbf{k}$.
We find, $$\mathfrak{g}_1 = \mathrm{span}_{\mathbb{C}}\left\lbrace \tilde{D}_1 := \tiny\left[\begin{array}{cccc} 0 & 0 & 1 & 0\\ 0 & 0 & 0 & 6/5 \\ 0 & 0 & 0 & 0\\ 0 & 0 & 0 & 0\end{array}\right]\right\rbrace$$ which is trivially abelian and therefore doesn't give any obstruction to integrability.
\subsection{Reduction of $\mathrm{(LVE^{2}_{\phi})}$}
We want now to put   the matrix $A_2$ of $\mathrm{(LVE^{2}_{\phi})}$ into a reduced form. First we reduce the diagonal blocks as indicated in section \ref{subsection: first partial reduction} using the partial reduction matrix $Q_2:=\tiny\left[\begin{array}{cccc}Sym^2 P_1 &  0 \\ 0 & P_1\end{array}\right]$ so that we obtain a partially reduced matrix (its diagonal blocks are reduced whereas its subdiagonal block is not):
$$
Q_2[A_2]:=\left[\begin{array}{cccc} sym^{2} A_{1,R} & 0 \\ \tilde{B}_2 & A_{1,R}\end{array}\right] \text{  with  } \left\lbrace\begin{array}{ccc} Q_2[A_2]_{diag} & = & \tiny\left[\begin{array}{cc} sym^2 A_{1,R} & 0 \\ 0 & A_{1,R}\end{array}\right]\\
Q_2[A_2]_{sub} & = & \tiny\left[\begin{array}{cc} 0 & 0 \\ \tilde{B}_2 & 0\end{array}\right]\end{array}\right\rbrace
$$
We compute a Wei-Norman decomposition and we obtain an associated Lie algebra $Lie(Q_2[A_2])$ of dimension $11$ such that:
\begin{itemize}
\item[-]  On one hand we obtain $Lie(Q_2[A_2])_{diag} = \mathrm{span}_{\mathbb{C}}\left\lbrace D_{2,0}:=\left[\begin{array}{cc} sym^2\tilde{D}_1 & 0 \\ 0 & \tilde{D}_1 \end{array}\right]\right\rbrace$ with coefficient $\beta_0:= \frac{5}{3 x}$.
\item[-] On the other hand, $Lie(Q_2[A_2])_{sub} = \mathrm{span}_{\mathbb{C}}\lbrace \mathcal{B}_2 \rbrace$ where
$$\mathcal{B}_2 := \lbrace B_i := {\tiny\left[\begin{array}{cc} 0 & 0 \\ \tilde{B}_i & 0\end{array}\right] , i=1\ldots 10\rbrace}\text{  and  }Q_2[A_2]_{diag} = \sum^{10}_{i=1} \beta_{2,i} B_{2,i}\text{  with  }\beta_i \in \mathbf{k}.$$
\end{itemize}

The matrix of the application  $$\Psi_{2,0}\,:\, Lie(Q_2[A_2])_{sub}\,\longrightarrow\, Lie(Q_2[A_2])_{sub} \,,\, B_j \,\mapsto\, [D_{2,0}\,,\, B_j]$$ expressed in the base $\mathcal{B}_2$  takes the following form:
$$
\Psi_{2,0}:=\tiny \left[ \begin {array}{cccccccccc} 0&0&0&0&0&0&0&0&1&0
\\\noalign{\medskip}-2&0&0&0&0&0&0&0&0&0\\\noalign{\medskip}0&0&0&0&0&0
&0&0&0&1\\\noalign{\medskip}0&0&-6/5&0&0&0&-1&0&0&0
\\\noalign{\medskip}0&-3&0&0&0&0&0&0&0&0\\\noalign{\medskip}0&0&0&-{
\frac {12}{5}}&0&0&0&-1&0&0\\\noalign{\medskip}0&0&0&0&0&0&0&0&0&6/5
\\\noalign{\medskip}0&0&0&0&0&0&-{\frac {12}{5}}&0&0&0
\\\noalign{\medskip}0&0&0&0&0&0&0&0&0&0\\\noalign{\medskip}0&0&0&0&0&0
&0&0&0&0\end {array} \right].
$$
We denote by $J_{\Psi_{2,0}}$ the matrix of $\Psi_{2,0}$ expressed in its  Jordan basis, given by the matrices $C_{2,i} =\tiny\left[\begin{array}{cc} 0 & 0 \\ \tilde{C}_{2,i} & 0 \end{array}\right]$ and their coefficients $\gamma_{2,i}$ with $i=1\ldots 10$.
So the Jordan form is
$$
J_{\Psi_{2,0}}=\tiny \left[ \begin {array}{cccccccccc} 0&1&0&0&0&0&0&0&0&0
\\\noalign{\medskip}0&0&1&0&0&0&0&0&0&0\\\noalign{\medskip}0&0&0&1&0&0
&0&0&0&0\\\noalign{\medskip}0&0&0&0&0&0&0&0&0&0\\\noalign{\medskip}0&0
&0&0&0&1&0&0&0&0\\\noalign{\medskip}0&0&0&0&0&0&1&0&0&0
\\\noalign{\medskip}0&0&0&0&0&0&0&1&0&0\\\noalign{\medskip}0&0&0&0&0&0
&0&0&0&0\\\noalign{\medskip}0&0&0&0&0&0&0&0&0&1\\\noalign{\medskip}0&0
&0&0&0&0&0&0&0&0\end {array} \right].
$$
To perform reduction we will use the Jordan basis $\mathcal{C}_2 := \lbrace C_{2,i} \,,\, i=1\ldots 10\rbrace$. The decomposition given by the Jordan basis $\mathcal{C}_2$ is $Q_2[A_2]:=D_0 + \sum^{10}_{i=1} \gamma_i C_i$ with $\gamma_i \in\mathbf{k}$ , $i=1\ldots10$. We notice that $J_{\Psi_0}$ is made of three Jordan blocks \begin{itemize}
\item[-] two blocks of dimension $4$ :
$$\lbrace C_{2,4}\,,\,C_{2,3}\,,\,C_{2,2}\,,\,C_{2,1} \rbrace$$ and  $$\lbrace C_{2,8}\,,\,C_{2,7}\,,\,C_{2,6}\,,\,C_{2,5} \rbrace$$
\item[-] and one block of dimension $2$ :  $\lbrace C_{2,10}\,,\,C_{2,9} \rbrace$
\end{itemize}
The hypothesis of the first section of Proposition \ref{partial reduction} are satisfied. Therefore the partial reduction of $Q_2[A_2]$ is done in the following way:
\begin{itemize}
\item[-] Choose a Jordan block of dimension $d$ : $\lbrace C_{2,i}\,\ldots\, C_{2,i+d-1}\rbrace$.
It satisfies $\Psi_{2,0}(C_{2,i+s}) = C_{2,i+s-1}$ for $s=1\ldots d-1$.
Set $\tilde{A}_{2}:=Q_2[A_2]$ and set $s:=d-1$.
\item[-] For $s$ from $d-1$ to $1$,
  compute the decomposition  $\gamma_{2,i+s}= R'_{2,i+s} + L_{2,i+s}$ where $R_{2,i+s}\,,\, L_{2,i+s}\in\mathbf{k}$  and $L_{2,i+s}$ has only simple poles.
  \\
  Take the change of variable $P_{2,i+s}=Id + R_{2,i+s} C_{2,i+s}$ and perform the gauge transformation $P_{2,i+s}[\tilde{A}_{2}]$.
  \\
  If $L_{2,i+s}=0$ then  the Wei-Normal decomposition of $P_{2,i+s}[\tilde{A}_{2}]$ does not contain $C_{2,i+s}$ so $C_{2,i+s}\notin\mathfrak{g}_2$.
%  \\
%  If $L_{2,i+s}\neq 0$ then $C_{2,i+s}\in\mathfrak{g}_{2}$.
	%%no : it depends whether L_{2,i+s} is a linear combination of the already computed coefficients.  If not, we still need a proof that
	%%  $C_{2,i+s}\in\mathfrak{g}_{2}$ (but I believe it is true).
  \\
  Set $\tilde{A}_{2} := P_{2,i+s}[\tilde{A}_{2}]$ and set $s:=s-1$. Repeat this procedure  recursively until $s=1$.
\item[-] Choose a Jordan block that has not been treated. Repeat until there are no more Jordan blocks left untreated.
\end{itemize}

In this way, only will be left in the subdiagonal block the $C_{2,i}$ that have coefficients $L_{2,i}$ (after the procedure) containing only simple poles. In our case, we obtain a reduced matrix for $\mathrm{(LVE^{2}_{\phi})}$: $A_{2,R}:=\frac{1}{x} \tilde{C}_0$ and

$$
 \tilde{C}_0:=\tiny\left[ \begin {array}{cccccccccccccc} 0&0&\frac53\, &0&0&0&0&0&0&0
&0&0&0&0\\\noalign{\medskip}0&0&0&2\, &0&\frac53\, &0&0&0&0&0
&0&0&0\\\noalign{\medskip}0&0&0&0&0&0&0&\frac{10}{3}\, &0&0&0&0&0&0
\\\noalign{\medskip}0&0&0&0&0&0&0&0&\frac53\, &0&0&0&0&0
\\\noalign{\medskip}0&0&0&0&0&0&2\, &0&0&0&0&0&0&0
\\\noalign{\medskip}0&0&0&0&0&0&0&0&2\, &0&0&0&0&0
\\\noalign{\medskip}0&0&0&0&0&0&0&0&0&4\, &0&0&0&0
\\\noalign{\medskip}0&0&0&0&0&0&0&0&0&0&0&0&0&0\\\noalign{\medskip}0&0
&0&0&0&0&0&0&0&0&0&0&0&0\\\noalign{\medskip}0&0&0&0&0&0&0&0&0&0&0&0&0&0
\\\noalign{\medskip}0&0&-\frac{10}{3}\, &0&0&0&2\, &\frac{95}{18} &0&-\frac{20}{3}\, &0&0&\frac53\, &0
\\\noalign{\medskip}0&0&0&0&0&2\, &0&0&-\frac{20}{3}&0&0&0&0&2\, \\\noalign{\medskip}0&0&0&0&0&0&0&\frac{10}{3}\,
&0&0&0&0&0&0\\\noalign{\medskip}0&0&0&0&0&0&0&0&-2\, &0&0&0&0&0
\end {array} \right]
$$

As in the case of $A_{1,R}$, this matrix  $A_{2,R}$ is in a reduced form because $Lie(A_{2,R})$ is monogenous and $\frac{1}{x}$ only has simple poles. Therefore  $Lie(A_{2,R}) =\mathfrak{g}_2$ and $\mathfrak{g}_2$ is once more abelian bringing in no obstruction to integrability. We have then to look at $\mathrm{(LVE^{3}_{\phi})}$.
\subsection{Reduction of $\mathrm{(LVE^{3}_{\phi})}$}
We denote $P_2$ the reduction matrix of $A_2$. Once more we build a partial reduction matrix $Q_3 := \tiny\left[\begin{array}{cc} Sym^3 P_1 & 0 \\ 0 & P_2\end{array}\right]$ that puts the diagonal blocks of matrix  $A_3$ into a reduced form and we obtain the partially reduced matrix $Q_3[A_3] := \tiny\left[\begin{array}{cc} sym^3 A_{1,R} & 0 \\ \tilde{B}_3 & A_{2,R}\end{array}\right]$. In this case we have a Wei-Norman decomposition of $Q_3[A_3]$ of dimension $18$, and $\mathrm{dim}_{\mathbb{C}}(Lie(Q_3[A_3]))=38$.

We thus have
\begin{itemize}
\item[-]$\mathrm{dim}_{\mathbb{C}}(Lie(Q_3[A_3])_{diag})=1$ where
$$Lie(Q_3[A_3])_{diag}=\mathrm{span}_{\mathbb{C}} \lbrace D_{3,0}:=\tiny\left[\begin{array}{cc} Sym^3 \tilde{D}_1 & 0 \\ 0 & \tilde{C}_{2,0}\end{array}\right]\rbrace$$
\item[-] and $\mathrm{dim}_{\mathbb{C}}(Lie(Q_3[A_3])_{sub})=37$ and $Lie(Q_3[A_3])_{sub} =\mathrm{span}_{\mathbb{C}}(\mathcal{B}_3)$ with
$$\mathcal{B}_3=\tiny\lbrace B_{3,i}=\left[\begin{array}{cc} 0 & 0 \\ \tilde{B}_{3,i} & 0\end{array}\right] \,,\,{\tiny i=1\ldots 38} \rbrace$$
a base of generators of $Lie(Q_3[A_3])_{sub}$.
\end{itemize}
We define $\Psi_{3,0} \, : \, \mathfrak{h}_{3,sub}\,\longrightarrow \, \mathfrak{h}_{3,sub}\,,\,
 B \, \mapsto \, [D_{3,0} \,,\, B]$.
 It is nilpotent and its Jordan basis will satisfy the conditions of the first section of Proposition \ref{partial reduction}.
 In the Jordan basis  $\mathcal{C}_{3}:=\lbrace C_{3,i}\,,\, i=1\ldots 37\rbrace$, the Jordan form of  $J_{\Psi_{3,0}}$ is formed by the following Jordan blocks:
\begin{enumerate}
\item three Jordan blocks of dimension $5$ corresponding to
: $\lbrace C_{3,5},\ldots , C_{3,1}\rbrace,$  $\lbrace C_{3,11},\ldots , C_{3,6}\rbrace,$ $\lbrace C_{3,17},\ldots , C_{3,12}\rbrace$
\item three Jordan blocks of dimension $4$:  $\lbrace C_{3,18},\ldots , C_{3,21}\rbrace$ ,  $\lbrace C_{3,22},\ldots , C_{3,26}\rbrace$  and  $\lbrace C_{3,31},\ldots , C_{3,27}\rbrace,$
\item and two Jordan blocks of dimension $2$:  $$\lbrace C_{3,34},\ldots , C_{3,32}\rbrace \text{ and }\lbrace C_{3,37},\ldots , C_{3,35}\rbrace.$$
\end{enumerate}
 In the basis $\mathcal{C}_3$, a Wei-Norman decomposition is $$Q_3[A_3]= \beta_{0} D_{3,0} + \sum^{37}_{i=1} \gamma_{3,i} C_{3,i}.$$

We proceed blockwise as in the case of the second variational equation. This time, possible obstructions to integrability appear when handling the Jordan block $\lbrace C_{3,31}\,,\ldots \,,\, C_{3,27}\rbrace$. By decomposition $\gamma_{3,i}=R'_{3,i} + L_{3,i}$ (with $i=27\ldots 31$), we see that in particular $L_{3,30}$ and $L_{3,29}$ are non zero (and have "new poles", i.e not the pole zero of the coefficient of the reduced form of $(VE_2)$) and therefore we suspect
that   $C_{3,29} , C_{3,30}$ (or some linear combination) lie in  $\mathfrak{g}_{3}$.
Since neither $C_{3,30}$ nor $C_{3,29}$ commute with $D_{3,0}$ that would suggest that $\mathfrak{g}_{3}$ is not abelian and therefore, intuitively, the Hamiltonian (\ref{HH}) would be non integrable. We prove this rigorously in the following subsection.

\subsection{Proof of non-integrability}
After performing the partial reduction recursively for all blocks,  we obtain the matrix $\tilde{A}_{3,R}$.
It has a  Wei-Norman decomposition  $\tilde{A}_{3,R} =a_1 M_{3,1} + a_2 M_{3,2}$ where $M_{3,1}, M_{3,2}\in\mathcal{M}_{34}(\mathbb{C})$, $a_1 :=\frac{1}{x}$, $a_2:=\frac{x}{x^2 +1}$. The matrix  $M_{3,1}$ is lower block triangular and   $M_{3,2}\in Lie(\tilde{A}_{3,R})_{sub}$. We let $M_{3,3}:=[M_{3,1}\,,\, M_{3,2}]$, $M_{3,4}:=[M_{3,1}\,,\, M_{3,3}]$, $M_{3,5}:=[M_{3,1}\,,\, M_{3,4}]$ and check that $[M_{3,i}\,,\, M_{3,j}]=0$ otherwise. So $Lie(\tilde{A}_{3,R})$ has dimension $5$  and is generated by the $M_{3,i}$. Note that $M_{3,i}\in \mathcal{M}_{34,sub}(\mathbb{C})$ for $i\geq 2$.   Again we let $$\Psi\,:\, Lie(\tilde{A}_{3,R})\,\longrightarrow\, Lie(\tilde{A}_{3,R})\quad,\quad M\mapsto [M_{3,1}\,,\, M].$$ By construction, the matrix of $\Psi$ is $\tiny\left[\begin{array}{ccccc} 0& 0 & 0 & 0 & 0\\ 0 & 0 & 0 & 0 & 0 \\ 0 & 1 & 0 & 0 & 0\\ 0 & 0 & 1 & 0 & 0\\ 0 & 0 & 0 & 1 &0\end{array}\right]$.
\begin{theorem}
$\tilde{A}_{3,R}$ is a reduced form for $\mathrm{(LVE^3_{\phi})}$ and $\mathfrak{g}_{3}$ is not abelian so the degenerate H\'enon-Heiles Hamiltonian (\ref{HH}) is not meromorphically integrable.
\end{theorem}
\begin{proof} We know that $Lie(\tilde{A}_{3,R})$ is non abelian so we just need to prove that $\tilde{A}_{3,R}$ is a reduced form. To achieve this we will construct a Picard Vessiot extension $K_{3}$ still using our ``reduction'' philosophy and we prove that  it has transcendence degree $5$: as $\mathfrak{g}_3\subset Lie(\tilde{A}_{3,R})$ and $\mathrm{dim}_{\mathbb{C}}(Lie(\tilde{A}_{3,R})) = 5$ this will show that $\mathfrak{g}_{3}=Lie(\tilde{A}_{3,R})$ because $\mathrm{dim}_{\mathbb{C}}(\mathfrak{g}_{3})=\mathrm{dtr}(K_3/ \mathbf{k})$ (see \cite{PuSi03a} Chap. 1.).

We apply proposition \ref{partial reduction}  to $\tilde{A}_{3,R}$. Apply the partial reduction $P_{1} = (Id + \int a_1 M_{3,1}) = Id+\ln(x) M_{3,1}$:  $P_1[\tilde{A}_{3,R}]$ contains no terms in $M_{3,2}$ and $ P_{1}[\tilde{A}_{3,R}] = a_1 M_{3,1} + \left(a_1 \int a_2\right) M_{3,3} $; we call $I_2 =\int ( a_1 \int a_2)  = Li_{2}(x^2)$ where $Li_2$ denotes the classical dilogarithm (see e.g \cite{Ca02a}). Similarly   we obtain $I_3$ and $I_4$  as coefficients of successive changes of variable. We are left with a system $Y'=a_1 M Y$, the Picard-Vessiot extension is
$$K_3= \mathbb{C}(x)(\ln(x)\,,\, \ln(1+x^2)\,,\, Li_{2}(x^2)\,,\, Li_{3}(x^2)\,,\, Li_{4}(x^2))$$

It is known to specialists that $\mathrm{dtr}(K_3/\mathbf{k})=5$ (and reproved for convenience below).
\end{proof}
\subsection{A self-contained proof of $\mathrm{dtr}(K_3/\mathbf{k}) =5$}\label{section: appendix}
To remain self-contained we propose a differential Galois theory proof of the following classical fact (see \cite{Ca02a} for instance). The proof is simple and beautifully consistent with our approach. To simplify the notations, we write the proof in the case of the classical iterated dilogarithms $Li_{j}(-x)$ but, of course, it applies mutatis mutandis to our case of $Li_{j}(x^2)$.
\begin{lemma}\label{appendix}
Let $K_3= \mathbb{C}(x)(\ln(x)\,,\, -\ln(1-x)\,,\, Li_{2}(-x)\,,\, Li_{3}(-x)\,,\, Li_{4}(-x))$, then $\mathrm{dtr}(K_3/\mathbf{k}) =5$
\end{lemma}
\begin{proof}
Let us prove that the functions $$x\,,\,\ln(x)\,,\, -\ln(1-x)\,,\, Li_{2}(-x)\,,\, Li_{3}(-x)\,,\, Li_{4}(-x)$$
are algebraically independent using a differential Galois theory argument.
That $\ln(x)$ and $-\ln(1-x)$ are transcendent and algebraically independent over $\mathbb{C}(x)$ is a classical  easy fact.  We focus in proving  the transcendence and algebraic independence of $Li_{2}(-x)\,,\, Li_{3}(-x)$ and $Li_{4}(-x)$.  Set the following relations,
$$
Li_0(-x) := \frac{x}{1-x},\quad Li_{1}(-x) := -\ln(1-x),\quad Li_{2}(-x):=\int\frac{Li_{1}(-x)}{x} dx ,
$$
$$
 Li_{3}(-x):=\int\frac{Li_{2}(-x)}{x} dx ,\quad Li_{4}(-x):=\int\frac{Li_{3}(-x)}{x} dx
$$
and therefore $K_3 = \mathbb{C}(x)(\ln(x)\,,\,Li_{0}(-x),\ldots ,Li_{4}(-x))$ is  a differential field (with $Li'_{i}(-x) = \frac{Li_{i-1}(-x)}{x}$). Of course, $\mathrm{dtr}(K_3/\mathbf{k})\leq 5$.  Let us define
$$
V := \mathrm{span}_{\mathbb{C}}\left\lbrace 1\,,\,\ln(x)\,,\, \frac{\ln(x)^2}{2}\,,\, \frac{\ln(x)^3}{6}\,,\, Li_1(-x)\,,\, Li_2 (-x)\,,\, Li_3 (-x)\,,\, Li_4 (-x)\right\rbrace
$$
and consider and element $\sigma \in Gal(K_3/\mathbf{k})$. As  $\sigma(\ln'(x)) = \sigma(\frac{1}{x}) = \frac{1}{x}=\ln'(x) $  there exists a constant $c_0\in\mathbb{C}$ such that $\sigma(\ln(x)) = c_0$. Similarly, we obtain that $\sigma(\ln(x)^2 /2) = \ln(x)^2 / 2 + c_0\ln(x)+c^2_0$  and  $\sigma(\ln(x)^3 /6) = \ln(x)^3 / 6 + c^2_0\ln(x)/2 +c_0\ln(x)^2 /2 c^3_0$. Since $Li'_{1}(-x) =\frac{x}{x^2 +1}\in\mathbf{k}$ we have that $\sigma(Li'_{1}(-x)) = Li'_{1}(-x)$ and therefore there exists $c_1\in\mathbb{C}$ such that $\sigma(Li_{1}(-x)) =Li_{1}(-x) + c_1$.  As $Li'_{2}(-x)  =\frac{Li_{1}(-1)}{x}$ we have that $\sigma(Li'_{2}(-x) ) =\sigma(\frac{Li_{1}(-1)}{x}) =\frac{Li_{1}(-x) }{x} + \frac{c_1}{x}$ and there exists $c_2\in\mathbb{C}$ such that $\sigma(Li_{2}(-x))=Li_{2}(-x) + c_1 \ln(x) +c_2$. We prove similarly  the existence of $ c_3 , c_4 \in\mathbb{C}$ such that
\begin{eqnarray}
\nonumber\sigma(Li_{3}(-x))&=&Li_{3}(-x) +c_1 \frac{\ln(x)^2}{2} + c_2 \ln(x) + c_3\\
\nonumber\sigma(Li_{4}(-x)) &=&Li_{4}(-x) + c_1\frac{\ln(x)^3}{6} + c_2 \frac{\ln(x)^2}{2} + c_3 \ln(x) + c_4.
\end{eqnarray}
We see that $V$ is stable under the action of $Gal(K_3/\mathbf{k})$ and hence is the solution space of a differential operator $L\in\mathbf{k}[\frac{d}{dx}]$ of order $8$. Therefore, in this basis the matrix of the action of  $\sigma$  on $V$:
$$
M_{\sigma}:=\tiny\left[\begin{array}{ccccccccc} 1 & c_0 & c^2_0 /2 & c^3_0/6 & c_1 & c_2 & c_ 3 & c_ 4 \\
											0 & 1 &c_0 & c^2_0 & 0 & c_1 & c_2 & c_3\\
											0 & 0 & 1 & c_0 & 0 & 0 & c_1 & c_2\\
											0 & 0 & 0 & 1 & 0 & 0 & 0 & c_1\\
											0 & 0 & 0 & 0 & 1 & 0 & 0 & 0\\
											0 & 0 & 0 & 0 & 0 & 1 & 0 & 0\\
											0 & 0 & 0 & 0 & 0 & 0 & 1 & 0\\
											0 & 0 & 0 & 0 & 0 & 0 & 0 & 1\end{array}\right]
$$
As $\ln(x)$ and $\ln(1-x)$ are transcendent (and algebraically independent) we know that $c_0$ and $c_1$ span $\mathbb{C}$. It follows that $\mathfrak{g}_3$ contains at least
$$
m_0 :={\tiny \left[ \begin{array}{cccccccc} 0 & 1 & 0 & 0 & 0 & 0 & 0 & 0\\
									0 & 0 & 1 & 0 & 0 & 0 & 0 & 0\\
									0 & 0 & 0 & 1 & 0 & 0 & 0 & 0\\
									0 & 0 & 0 & 0 & 0 & 0 & 0 & 0\\
									0 & 0 & 0 & 0 & 0 & 0 & 0 & 0\\
									0 & 0 & 0 & 0 & 0 & 0 & 0 & 0\\
									0 & 0 & 0 & 0 & 0 & 0 & 0 & 0\\
									0 & 0 & 0 & 0 & 0 & 0 & 0 & 0																							 
								\end{array}\right]}  \quad \text{and}\quad 		
m_1 :={\tiny \left[ \begin{array}{cccccccc}
									0 & 0 & 0 & 0 & 1 & 0 & 0 & 0\\
									0 & 0 & 0 & 0 & 0 &1  & 0 & 0\\
									0 & 0 & 0 & 0 & 0 & 0 & 1 & 0\\
									0 & 0 & 0 & 0 & 0 & 0 & 0 & 1\\
									0 & 0 & 0 & 0 & 0 & 0 & 0 & 0\\
									0 & 0 & 0 & 0 & 0 & 0 & 0 & 0\\
									0 & 0 & 0 & 0 & 0 & 0 & 0 & 0\\
									0 & 0 & 0 & 0 & 0 & 0 & 0 & 0																							 
								\end{array}\right]}.
$$
Since $m_0$ and $m_1$ do not commute, we know that the Lie algebra generated by the iterated Lie brackets has dimension at least $3$. Iterating the brackets of $m_0$ and $m_1$ we obtain a subalgebra of $\mathfrak{g}_3$ of dimension $5$.
 Therefore we have $\mathrm{dtr}(K_{3}/ \mathbf{k}) \geq 5$ and since we know that $\mathrm{dtr}(K_{3}/ \mathbf{k})\leq 5$ we obtain the equality and the result follows.

\end{proof}

\begin{remark}
Horozov and Stoyanova \cite{HS07} make use of the properties of the dilogarithm in order to prove the non-integrability of some subfamilies of Painlev\'e VI equations: namely, they prove the non-abelianity of $\mathfrak{g}_2$, the Lie algebra of its second variational equation.
\end{remark}
\section{Conclusion}
The reduction method proposed here is systematic (and we have implemented it in Maple). Although it is currently limited to the case when $Lie(\mathrm{(VE^1_{\phi})})$ is one-dimensional, extensions to higher dimensional cases along the same guidelines are in progress and will appear in subsequent work.
In work in progress with S. Simon, we will show another use of reduced forms, namely the expression of taylor expansions of first integrals along $\phi$ are then greatly simplified. \\
We conjecture that our method is not only a partial reduction procedure but a complete reduction algorithm : assuming that $\mathrm{(LVE^m_{\phi_0})}$ is reduced (with an abelian Lie algebra), we believe that the output $\tilde{A}_{m+1,R}$ of our reduction procedure of sections 4 and 5 will always be a reduced form.
In the context of complex Hamiltonian systems, this would mean that our method would lead to an  effective version of the Morales-Ramis-Sim\'o theorem.

\bibliographystyle{amsalpha}

\end{document}